\documentclass[11pt, a4paper]{amsart}

\usepackage{amsmath, amsthm, amssymb}
\theoremstyle{plain}
\newtheorem{theorem}{Theorem}[section]

\newtheorem{proposition}[theorem]{Proposition}
\newtheorem{mainth}{Theorem}
\newtheorem{maincoro}{Corollary}
\theoremstyle{definition}

\newtheorem*{acknowledgement}{Acknowledgement}
\theoremstyle{remark}

\newcommand{\sD}{\mathsf{D}}
\newcommand{\sP}{\mathsf{P}}
\newcommand{\sA}{\mathsf{A}}
\newcommand{\bA}{\mathbb{A}}
\newcommand{\cO}{\mathcal{O}}
\newcommand{\can}{\operatorname{can}}
\newcommand{\cS}{\mathcal{S}}
\newcommand{\diam}{\operatorname{diam}}
\newcommand{\bP}{\mathbb{P}}
\newcommand{\sF}{\mathsf{F}}
\newcommand{\bR}{\mathbb{R}}
\newcommand{\bN}{\mathbb{N}}
\newcommand{\sH}{\mathsf{H}}
\newcommand{\cM}{\mathcal{M}}
\newcommand{\Vu}{\boldsymbol{u}}
\newcommand{\Vv}{\boldsymbol{v}}

\numberwithin{equation}{section}

\begin{document} 

\title[Lehto--Virtanen-type and big Picard-type theorems]{Lehto--Virtanen-type and big Picard-type theorems for Berkovich analytic spaces}
\author[Y\^usuke Okuyama]{Y\^usuke Okuyama}
\address{
Division of Mathematics,
Kyoto Institute of Technology,
Sakyo-ku, Kyoto 606-8585 Japan.}
\email{okuyama@kit.ac.jp}

\date{\today}

\begin{abstract} 
In non-archimedean setting,
we establish a Lehto--Virtanen-type theorem for 
a morphism from the punctured Berkovich closed unit disk $\overline{\sD}\setminus\{0\}$
in the Berkovich affine line to the Berkovich projective line $\sP^1$
having an isolated essential singularity at the origin, and then
establish a big Picard-type theorem for such an open subset $\Omega$ in 
the Berkovich projective space $\sP^N$ of any dimension $N$
that the family of all morphisms from $\overline{\sD}\setminus\{0\}$
to $\Omega$ is normal in a non-archimedean Montel's sense.
As an application of the latter theorem, we see a 
big Brody-type hyperbolicity of the Berkovich harmonic 
Fatou set of an endomorphism of $\sP^N$ of degree $>1$.
\end{abstract}


\maketitle

\section{Introduction}\label{sec:intro}

Let $K$ be a field of any characteristic
that is complete with respect to a {\em non-archimedean} 
absolute value $|\cdot|$.
Let $\overline{\sD}$ be the Berkovich closed unit disk 
in the Berkovich affine line $\sA^1=\sA^1(K)=\bA^1(K)^{\operatorname{an}}$. 
We note that 
\begin{gather*}
 \overline{\sD}\cap K=\cO_K=\{z\in K:|z|\le 1\},
\end{gather*}
and that $\overline{\sD}=\sP^1\setminus U(\overrightarrow{\cS_{\can}\infty})$
in the notation in \S\ref{sec:berkovich}.
We say a morphism $f$ from $\overline{\sD}\setminus\{0\}$
to a Berkovich $K$-analytic space $X$ {\em has an isolated
essential singularity at the origin} if $f$ does not extend to 
a morphism from $\overline{\sD}$ to any Berkovich $K$-analytic space.
One of our aims is to see the following
non-archimedean analog of Lehto--Virtanen and Lehto 
\cite{LV57Fennicae,Lehto59}
(see also \cite{Okuhyperbolic,OPrescaling}).

\begin{mainth}[a Lehto--Virtanen-type theorem]\label{th:LV}
Let $K$ be a field of any characteristic
that is complete with respect to a non-archimedean 
absolute value $|\cdot|$. Then for every morphism
$f$ from $\overline{\sD}\setminus\{0\}$ to the Berkovich projective line
$\sP^1=\sP^1(K)$ 
having an isolated essential singularity at the origin, we have
\begin{gather}
\limsup_{r\searrow 0}\diam_\#\bigl(f(\{z\in K:|z|=r\})\bigr)=\diam_\#(\bP^1).\label{eq:LV}
\end{gather}
Here, $\diam_\#$ is the chordal diameter function
on $(2^{\bP^1})\setminus\{\emptyset\}$ with respect 
to an equipped chordal metric on $\bP^1=\bP^1(K)$.
\end{mainth}

In the proof of Theorem \ref{th:LV}, 
we will normalize the equipped chordal metric on $\bP^1$ as
$\diam_\#(\bP^1)=1$.
The proof of Theorem \ref{th:LV} 
is an improvement of some
argument in the proof of Rodr{\'\i}guez V{\'a}zquez
\cite[Proposition 7.17]{Rodriguez16}, which was a little
Picard-type theorem.
An argument similar to that in the
proof of Theorem \ref{th:LV} also yields the following big version of
\cite[Proposition 7.17]{Rodriguez16}.
For the definition of the non-archimedean Montel-type normality
appearing in Theorem \ref{th:big}, 
see \cite[\S 1, \S 7]{Rodriguez16} and \cite[Introduction]{FKT12}.

\begin{mainth}[a big Picard-type theorem]\label{th:big}
Let $K$ be a field of any characteristic
that is complete with respect to a non-archimedean 
absolute value, and 
let $\Omega$ be an open subset in the Berkovich projective space $\sP^N=\sP^N(K)$
of any dimension $N$
such that the family $\operatorname{Mor}(\overline{\sD}\setminus\{0\},\Omega)$
of all morphisms from $\overline{\sD}\setminus\{0\}$ to $\Omega$ is normal.
Then any morphism from $\overline{\sD}\setminus\{0\}$ to $\Omega$ extends
to a morphism $\overline{\sD}\to\sP^N$.
\end{mainth}

By Rodr{\'\i}guez V{\'a}zquez \cite[Theorem C]{Rodriguez16},
an example of such an $\Omega$ as in Theorem \ref{th:big}
is a component of the (Berkovich) {\em harmonic Fatou set} 
$\sF_{\operatorname{harm}}(f)$ of an endomorphism $f$ of $\sP^N$ 
of degree $>1$; see \cite[Definition 7.9]{Rodriguez16}
for the definition of the (Berkovich) harmonic Fatou set 
of $f$. Hence we conclude the following.

\begin{maincoro}[a big Brody-type hyperbolicity of the harmonic Fatou set]
Let $K$ be a field of any characteristic
that is complete with respect to a non-archimedean 
absolute value, and let $f$ be an endomorphism of $\sP^N=\sP^N(K)$ 
of any dimension $N$ of degree $>1$. Then
any morphism from $\overline{\sD}\setminus\{0\}$ to the $($Berkovich$)$ 
harmonic Fatou set $\sF_{\operatorname{harm}}(f)$
of $f$ extends to a morphism $\overline{\sD}\to\sP^N$.
\end{maincoro}

The proof of \cite[Proposition 7.17]{Rodriguez16} 
invoked the 
non-archimedean {\em little} Picard theorem,
which asserts that {\em any $K$-analytic mapping $f:\bA^1\to\bP^1$ 
satisfying $\#(\bP^1\setminus f(\bA^1))\ge 2$ is constant} (see, e.g., 
\cite[(1.3) Proposition]{vdP81}).
Our argument in the proofs of Theorems \ref{th:LV} and \ref{th:big}
instead requires a Riemann-type extension theorem
(near an isolated singularity) below, which is almost
straightforward from the Laurent expansion of a $K$-analytic
function around an isolated singularity of it
and the strong triangle inequality,
and which can also be adopted to give a more elementary proof of \cite[Proposition 7.17]{Rodriguez16}.

\begin{proposition}[a Riemann-type extension theorem]\label{th:Riemann}
Let $f$ be a $K$-analytic function on $\cO_K\setminus\{0\}$.
If $|f|$ is bounded near $0$, then $f$ extends to a $K$-analytic function
on $\cO_K$. In particular, $f$ extends to a morphism $\overline{\sD}\to\sP^1$.
\end{proposition}

\section{Background}
\label{sec:background}
Let $K$ be a field of any characteristic
that is complete with respect to a {\em non-archimedean} 
absolute value $|\cdot|$. Recall that the absolute value 
$|\cdot|$ is said to be non-archimedean if 
the {\em strong} triangle inequality
$|z+w|\le\max\{|z|,|w|\}$ holds for any $z,w\in K$.
Let $\pi=\pi_N:K^{N+1}\setminus\{(0,\ldots,0)\}\to\bP^N=\bP^N(K)$ 
be the canonical projection associated to $\bP^N$ of any dimension $N$, and
let $\|\cdot\|=\|\cdot\|_\ell$
be the {\itshape maximal} norm 
$\|(z_1,\ldots,z_\ell)\|=\max\{|z_1|,\ldots,|z_\ell|\}$ on $K^\ell$
of any dimension $\ell$.
Noting that $\bigwedge^2K^{N+1}\cong K^{\binom{N+1}{2}}$ as $K$-linear spaces
(cf.\ \cite[\S 8.1]{Kobayashi98}),
the (normalized) {\em chordal metric} on $\bP^N$ is defined as
\begin{gather*}
 [z,w]_{\bP^N}:=
 \frac{\|Z\wedge W\|_{\binom{N+1}{2}}}{\|Z\|_{N+1}\cdot\|W\|_{N+1}},\quad z,w\in\bP^N,
\end{gather*}
where $Z\in\pi^{-1}(z),W\in\pi^{-1}(w)$
(the notation is adopted from
Nevanlinna's and Tsuji's books \cite{Nevan70,Tsuji59}),
so that $\diam_\#(\bP^N)=1$. We equip $\bP^1=K\cup\{\infty\}$ with this normalized $[z,w]_{\bP^1}$ in this section. The topology of $\bP^1$
coincides with the metric topology of $(\bP^1,[z,w]_{\bP^1})$.

\subsection{Berkovich projective line as a tree}\label{sec:berkovich}
For the details on $\sP^1$, see \cite{BR10,FR09}.
For simplicity, we also assume that 
$K$ is algebraically closed and $|\cdot|$ is non-trivial.
As a set, the Berkovich affine line $\sA^1=\sA^1(K)$ is the set of all
multiplicative seminorms on $K[z]$ extending $|\cdot|$.
An element of $\sA^1$ is denoted by $\cS$, and
also by $[\cdot]_{\cS}$ as a multiplicative
seminorm on $K[z]$. The topology of $\sA^1$ is
the weakest topology such that for any $\phi\in K[z]$, 
the function $\sA^1\ni\cS\mapsto[\phi]_{\cS}\in\bR_{\ge 0}$ is continuous,
and then $\sA^1$ is a locally compact, uniquely arcwise connected,
Hausdorff topological space. 
A subset $B$ in $K$ is called a {\em $K$-closed disk} if
\begin{gather*}
 B=B(a,r):=\{z\in K:|z-a|\le r\} 
\end{gather*}
for some $a\in K$ and some $r\ge 0$. 
For any $K$-closed disks $B,B'$, {\em if $B\cap B'\neq\emptyset$, then 
either $B\subset B'$ or $B'\subset B$}
(by the strong triangle inequality). 
The Berkovich representation \cite{Berkovichbook} asserts that
{\em any element $\cS\in\sA^1$ is 
induced by a non-increasing and nesting sequence $(B_n)$ of
$K$-closed disks in that
$[\phi]_{\cS}=\inf_{n\in\bN}\sup_{z\in B_n}|\phi(z)|$
for any $\phi\in K[z]$}; a point $\cS\in\sA^1$ is said to be of {\em type}
I, II, III, and IV 
if $\cS$ can be induced by a (constant sequence of a singleton 
$B(a,0)=\{a\}$ consisting of a unique) point $a\in\bA^1$, 
a (constant sequence of a) $K$-closed disk $B(a,r)$
satisfying $r\in|K^*|$, a (constant sequence of a) $K$-closed disk $B(a,r)$
satisfying $r\in\bR_{>0}\setminus|K^*|$, and any other case holds, respectively.
We identify, as a set, $K$ with
the set of all type I points in $\sA^1$.

Any $[\cdot]_{\cS}\in\sA^1$ extends to the function
$K(z)\to\bR_{\ge 0}\cup\{+\infty\}$ such that, 
for any $\phi=\phi_1/\phi_2\in K(z)$ where $\phi_1,\phi_2\in K[z]$ are coprime, 
$[\phi]_{\cS}=[\phi_1]_{\cS}/[\phi_2]_{\cS}\in\bR_{\ge 0}\cup\{+\infty\}$, 
and we also regard $\infty\in\bP^1$ as the function 
$[\cdot]_{\infty}:K(z)\to\bR_{\ge 0}\cup\{+\infty\}$ 
such that for every $\phi\in K(z)$,
$[\phi]_{\infty}=|\phi(\infty)|\in\bR_{\ge 0}\cup\{+\infty\}$.
As a set, the Berkovich projective line $\sP^1=\sP^1(K)$ 
is nothing but $\sA^1\cup\{\infty\}$. The point $\infty$ is
also said to be of type I. 

Set $\sH^1:=\sP^1\setminus\bP^1$, 
and let $\sH^1_{\mathrm{II}}$ (resp.\ $\sH^1_{\mathrm{III}}$)
be the set of all type II (resp.\ type III) points in $\sP^1$. 
The {\itshape Gauss $($or canonical$)$ point} 
\begin{gather*}
 \cS_{\can}\in\sH^1_{\mathrm{II}} 
\end{gather*}
is induced by the (constant sequence of the) $K$-closed disk
$\cO_K=B(0,1)$, that is, the ring of $K$-integers. Let $\cM_K$ be
the unique maximal ideal of $\cO_K$ and $k$ be the residue field
$\cO_K/\cM_K$ of $K$.

An ordering $\le_\infty$ on $\sP^1$ is defined
so that for any $\cS,\cS'\in\sP^1$, $\cS\le_\infty\cS'$
if and only if $[\cdot]_{\cS}\le_\infty[\cdot]_{\cS'}$ on $K[z]$.
For any $\cS,\cS'\in\sP^1$, if $\cS\le_{\infty}\cS'$, then set 
$[\cS,\cS']=[\cS',\cS]:=\{\cS''\in\sP^1:\cS\le_\infty\cS''\le_\infty\cS'\}$,
and in general, 
there is the unique point, say, $\cS\wedge_\infty\cS'\in\sP^1$ such that
$[\cS,\infty]\cap[\cS',\infty]=[\cS\wedge_\infty\cS',\infty]$, and set 
\begin{gather*}
 [\cS,\cS']:=[\cS,\cS\wedge_\infty\cS']\cup[\cS\wedge_\infty\cS',\cS'].
\end{gather*}
These {\em closed intervals} $[\cS,\cS']\subset\sP^1$ make $\sP^1$ 
an ``$\bR$-''tree in the sense of Jonsson \cite[Definition 2.2]{Jonsson15}.
For any $\cS\in\sP^1$, the equivalence class 
$T_{\cS}\sP^1:=(\sP^1\setminus\{\cS\})/\sim$ is defined so that
for any $\cS',\cS''\in\sP^1\setminus\{\cS\}$, 
$\cS'\sim\cS''$ if $[\cS,\cS']\cap[\cS,\cS'']=[\cS,\cS'\wedge_{\cS}\cS'']$ 
for some point say $\cS'\wedge_{\cS}\cS''\in\sP^1\setminus\{\cS\}$. 
An element of $T_{\cS}\sP^1$ is called a {\itshape direction} of $\sP^1$
at $\cS$ and denoted by $\Vv$, and also by $U(\Vv)=U_{\cS}(\Vv)$ as a
subset in $\sP^1\setminus\{\cS\}$. If $\Vv\in T_{\cS}\sP^1$ is 
represented by an element $\cS'\in\sP^1\setminus\{\cS\}$,
then we also write $\Vv$ as $\overrightarrow{\cS\cS'}$.
A point $\cS\in\sP^1$ is of type either I or IV if and only if
$\#T_{\cS}\sP^1=1$, that is, $\cS$ is an end point of $\sP^1$ as a tree. 
On the other hand, a point $\cS\in\sP^1$ is of type II (resp.\ type III)
if and only if $\#T_{\cS}\sP^1>2$ (resp.\ $=2$). 
A (Berkovich) {\em strict connected open affinoid} in $\sP^1$ is
a non-empty subset in $\sP^1$ which is
the intersection of some finitely many elements of 
$\{U(\Vv):\cS\in\sH^1_{\mathrm{II}},\Vv\in T_{\cS}\sP^1\}$.
The topology of $\sP^1$ 
has the quasi-open basis $\{U(\Vv):\cS\in\sP^1,\Vv\in T_{\cS}\sP^1\}$, 
so has the open basis consisting of all 
Berkovich strict connected open affinoids in $\sP^1$.
Both $\bP^1$ and $\sH^1_{\mathrm{II}}$ are dense in $\sP^1$,
the set $U(\Vv)$ is a component of $\sP^1\setminus\{\cS\}$
for each $\cS\in\sP^1$ and each $\Vv\in T_{\cS}\sP^1$,
and for any $\cS,\cS'\in\sP^1$,
the interval $[\cS,\cS']$ is the unique arc in $\sP^1$ between $\cS$ and $\cS'$.

We also denote the {\em left-half open} interval
$[\cS,\cS']\setminus\{\cS\}\subset\sP^1$ by $(\cS,\cS']$.
For every $0<r\le 1$, letting $\cS(r)\in(0,\cS_{\can}]$ be the point
induced by the (constant sequence of the) $K$-closed disk $B(0,r)$,
we have 
\begin{gather}
 \{z\in K:|z|=r\}=\bigcup_{\Vv\in T_{\cS(r)}\sP^1\setminus\{\overrightarrow{\cS(r)0},\overrightarrow{\cS(r)\infty}\}}(U(\Vv)\cap\bP^1).\label{eq:annulus}
\end{gather}
The normalized chordal metric $[z,w]_{\bP^1}$ on $\bP^1$ 
extends to an upper semicontinuous and separately continuous function
$(\cS,\cS')\mapsto[\cS,\cS']_{\can}$ on $\sP^1\times\sP^1$,
which is called the {\em generalized Hsia kernel} function on $\sP^1$
with respect to $\cS_{\can}$ (\cite[\S4.4]{BR10});
the function 
$\cS\mapsto[\cS,\cS]_{\can}$ is {\em continuous on any interval} in $\sP^1$,
and for every $\cS\in\sP^1$ and every $\Vv\in T_{\cS}\sP^1$, we have
\begin{gather}
\diam_{\#}\bigl(U(\Vv)\cap\bP^1\bigr)
=\begin{cases}
  [\cS_{\can},\cS_{\can}]_{\can}=1\quad\text{if }\cS=\cS_{\can},\\
  [\cS_{\can},\cS_{\can}]_{\can}=1\quad\text{if }\cS\neq\cS_{\can}
  \text{ and }\Vv=\overrightarrow{\cS\cS_{\can}}, \\
[\cS,\cS]_{\can}\quad\text{if }\cS\in(\sH^1_{\mathrm{II}}\setminus\{\cS_{\can}\})
\cup\sH^1_{\mathrm{III}}
\text{ and }\Vv\neq\overrightarrow{\cS\cS_{\can}}.
 \end{cases}
\label{eq:diameter}
\end{gather}

\subsection{Mapping properties of morphisms}
Any non-constant morphism $f$ from an open neighborhood of a point $\cS\in\sP^1$
to $\sP^1$ is finite to one {\em near $\cS$} and induces a {\em surjection} 
$f_*=(f_*)_{\cS}:T_{\cS}\sP^1\to T_{f(\cS)}\sP^1$,
which is called the {\em tangent map} of $f$ at $\cS$; 
when $f(\cS_{\can})=\cS_{\can}$,
$(f_*)_{\cS_{\can}}$ is regarded as the action on $\bP^1(k)$
of the reduction $\tilde{f}\in k(z)$ of $f$
(see, e.g., \cite[\S 2.6, \S 4.5]{Jonsson15} for the details).

Let $f:\overline{\sD}\setminus\{0\}\to\sP^1$ be a non-constant morphism.
Then from a general mapping property of a non-constant $K$-analytic mapping from 
a disk in $\bA^1$ to $\bP^1$ (see, e.g., \cite[\S 3]{Benedetto08}),
for every $\cS\in(0,\cS_{\can}]$ and every 
$\Vv\in T_{\cS}\sP^1\setminus\{\overrightarrow{\cS0},\overrightarrow{\cS\infty}\}$,
\begin{gather}
\text{if }f(U_{\cS}(\Vv))\neq\sP^1,\text{ then }
f(U_{\cS}(\Vv))= U_{f(\cS)}(f_*\Vv).
\label{eq:mapping}
\end{gather}

\section{Proofs of Theorems \ref{th:LV} and \ref{th:big}}\label{sec:LV}

With no loss of generality, we also assume that 
$K$ is algebraically closed and $|\cdot|$ is non-trivial.
We equip $\bP^1$ with the normalized chordal metric $[z,w]_{\bP^1}$
defined in Section \ref{sec:background},
so $\diam_\#(\bP^1)=1$.

\begin{proof}[Proof of Theorem \ref{th:LV}]
 Let $f:\overline{\sD}\setminus\{0\}\to\sP^1$ be a morphism, and
 suppose that $f$ does not satisfy \eqref{eq:LV}.
 Then by \eqref{eq:annulus}, 
 for every $\cS\in(0,\cS_{\can}]$ close enough to $0$ and every 
 $\Vv\in T_{\cS}\sP^1\setminus\bigl\{\overrightarrow{\cS 0},\overrightarrow{\cS\infty}\bigr\}$, we have $f(U(\Vv))\neq\sP^1$, 
 and in turn by \eqref{eq:diameter},
 \eqref{eq:mapping}, and the continuity of $f$,
 there is $\Vu_0\in T_{\cS_{\can}}\sP^1$ such that $f(\cS)\subset U(\Vu_0)$
 for every $\cS\in(0,\cS_{\can}]$ close enough to $0$.
 Then, under the assumption that $f$ does not satisfy \eqref{eq:LV}, 
 by \eqref{eq:diameter} and \eqref{eq:mapping},
 for any $\cS\in(0,\cS_{\can}]$ close enough to $0$, we even have
 \begin{gather*}
  f\bigl(\sP^1\setminus(U(\overrightarrow{\cS 0})\cup U(\overrightarrow{\cS\infty}))\bigr)
\subset U(\Vu_0),
 \end{gather*} 
 and then by \eqref{eq:annulus} and a Riemann-type extension theorem
 (Proposition \ref{th:Riemann}),
 $f$ extends to a morphism from $\sD$ to a $K$-analytic space.
\end{proof}

\begin{proof}[Proof of Theorem \ref{th:big}]
We can assume $N=1$ by an argument similar to that in the final paragraph in 
\cite[Proof of Proposition 7.17]{Rodriguez16} 
involving not only the existence of a nice lifting of $f|(\cO_K\setminus\{0\})$
to a morphism $\cO_K\setminus\{0\}\to\bA^{N+1}$ through
the canonical projection $\pi:\bA^{N+1}\setminus\{(0,\ldots,0)\}\to\bP^N$
(using \cite[Theorem 2.7.6]{FvdP04}) but also
the projection $\bA^{N+1}\setminus\{z_0=0\}\mapsto[z_0:z_1]\in\bP^1$.

Let $\Omega$ be an open subset in $\sP^1$.
Suppose that $\operatorname{Mor}(\overline{\sD}\setminus\{0\},\Omega)$ is normal
and, to the contrary, that
there is a morphism $f:\overline{\sD}\setminus\{0\}\to\Omega$
having an isolated essential singularity at the origin.

 {\bfseries (i).}
 If there is a point $\cS_0\in\sH^1_{\mathrm{II}}$
 such that $\#(f^{-1}(\cS_0)\cap(0,\cS_{\can}])=\infty$,
 then we can conclude a contradiction
 by an argument similar to that in the former half of 
 \cite[Proof of Theorem 7.17]{Rodriguez16}.
 For completeness, we include the argument; by the surjectivity
 of $f_*:T_{\cS}\sP^1\to T_{\cS_0}\sP^1$ for every 
 $\cS\in f^{-1}(\cS_0)\cap(0,\cS_{\can}]$ (and by $\#T_{\cS_0}\sP^1=\infty>2$),
 there are a sequence $(\cS_n)$ in $f^{-1}(\cS_0)\cap(0,\cS_{\can}]$  
 tending to $0$ as $n\to\infty$ and
 a direction $\Vu_0\in T_{\cS_0}\sP^1$ such that for every $n\in\bN$, 
 $\Vu_0\in f_*(T_{\cS_n}\sP^1\setminus\bigl\{\overrightarrow{\cS_n0},
 \overrightarrow{\cS_n\infty}\bigr\})$.
 Then fixing a point $a_0\in U(\Vu_0)\cap\bP^1$, for every $n\in\bN$,
 there is a point $b_n\in\bP^1\cap f^{-1}(a_0)$ such that
 $\overrightarrow{\cS_nb_n}\in T_{\cS_n}\sP^1\setminus
 \bigl\{\overrightarrow{\cS_n0},\overrightarrow{\cS_n\infty}\bigr\}$.
For every $n\in\bN$, by \eqref{eq:annulus}, $\cS_n$ is induced by the
 (constant sequence of the) $K$-closed disk $B(0,|b_n|)$. 

 Now setting 
\begin{gather*}
  g_n(z):=f(b_{n!}\cdot z^{n!})\in\operatorname{Mor}(\overline{\sD}\setminus\{0\},\Omega) 
\end{gather*}
for each $n\in\bN$, under the assumption that
 $\operatorname{Mor}(\overline{\sD}\setminus\{0\},\Omega)$ is {\em normal}
 in the sense of \cite[Introduction]{FKT12},
 taking a subsequence of $(g_n)$ if necessary, 
 the ({\em pointwise}) limit $g:=\lim_{n\to\infty}g_n$ on 
 $\overline{\sD}\setminus\{0\}$
 exists and is a {\em continuous} mapping $\overline{\sD}\setminus\{0\}\to\sP^1$.
 Then 
\begin{gather*}
  g(\cS_{\can})=\lim_{n\to\infty}g_n(\cS_{\can})
 =\lim_{n\to\infty}f(\cS_{n!})=\cS_0.
\end{gather*}
On the other hand, we can fix a sequence $(\zeta_m)$ in $K$ such that
$(\zeta_m)^m=1$ (so $\zeta_m\in\overline{\sD}\setminus\{0\}$)
for every $m\in\bN$ and that
 $\lim_{m\to\infty}\zeta_m=\cS_{\can}$. Then for any $m\in\bN$,
\begin{gather*}
  g(\zeta_m)=\lim_{n\to\infty}g_n(\zeta_m)
 =\lim_{n\to\infty}f(b_{n!}\cdot 1^{n!/m})=\lim_{n\to\infty}f(b_{n!})
=a_0,\quad\text{and in turn}\\
 g(\cS_{\can})=\lim_{m\to\infty}g(\zeta_m)=\lim_{m\to\infty}a_0=a_0.
\end{gather*} 
Hence 
we must have 
$\bP^1\ni a_0=\cS_0\in\sH^1_{\mathrm{II}}$,
which is a contradiction.

 {\bfseries (ii).}
 If there is a sequence $(\cS_n)$ in $(0,\cS_{\can}]$ tending to $0$
 as $n\to\infty$ such that for every $n\in\bN$, 
 there is a direction $\Vv_n\in T_{\cS_n}\sP^1\setminus
 \bigl\{\overrightarrow{\cS_n0},\overrightarrow{\cS_n\infty}\bigr\}$
 satisfying $f(U(\Vv_n))=\sP^1$, then 
 taking a subsequence of $(\cS_n)$ if necessary,
 there is a direction $\Vu_0\in T_{\cS_{\can}}\sP^1$ such that
 $f(\cS_n)\in\sP^1\setminus U(\Vu_0)$ for any $n\in\bN$.
 Fixing a point $a_0\in\bP^1\cap U(\Vu_0)$,
 there is a point $b_n\in\bP^1\cap U(\Vv_n)\cap f^{-1}(a_0)$
 for each $n\in\bN$.

 Now by a (branched) rescaling argument similar to that in the case (i),
 we must have $U(\Vu_0)\ni a_0\in\overline{\{f(\cS_n):n\in\bN\}}\subset\sP^1\setminus U(\Vu_0)$. 
 This is a contradiction. 

 {\bfseries (iii).}
 Suppose finally that for every $\cS\in\sH^1_{\mathrm{II}}$,
 $\#(f^{-1}(\cS)\cap(0,\cS_{\can}])<\infty$ and that 
 for every $\cS\in(0,\cS_{\can}]$ close enough to $0$
 and every $\Vv\in T_{\cS}\sP^1\setminus\{\overrightarrow{\cS0},\overrightarrow{\cS\infty}\}$, $f(U(\Vv))\neq\sP^1$.
 
 Then under the former assumption (for $\cS=\cS_{\can}$), by the continuity of $f$, 
 there is $\Vu_0\in T_{\cS_{\can}}\sP^1$
 such that $f(\cS)\subset U(\Vu_0)$
 for every $\cS\in(0,\cS_{\can}]$ close enough to $0$.
 Then under the latter assumption, by \eqref{eq:annulus},
 \eqref{eq:mapping}, and
 a Riemann-type extension theorem (Proposition \ref{th:Riemann}), 
 there must exist a sequence 
 $(\cS_n)$ in $(0,\cS_{\can}]$ tending to $0$ as $n\to\infty$ such that 
 for every $n\in\bN$, $f(\cS_n)\in U(\Vu_0)$ and
there is a direction 
$\Vv_n\in T_{\cS_n}\sP^1\setminus\bigl\{\overrightarrow{\cS_n0},\overrightarrow{\cS_n\infty}\bigr\}$ satisfying
\begin{gather*}
 f_*(\Vv_n)=\overrightarrow{f(\cS_n)\cS_{\can}},
\end{gather*} 
and then $f(U(\Vv_n))\supset\sP^1\setminus U(\Vu_0)$.
Fixing a point $a_0\in\bP^1\setminus U(\Vu_0)=\bP^1\setminus\overline{U(\Vu_0)}$,
 there is a point $b_n\in\bP^1\cap U(\Vv_n)\cap f^{-1}(a_0)$
 for each $n\in\bN$.

 Now by a (branched) rescaling argument similar to that in the case (i),
 we have $\bP^1\setminus\overline{U(\Vu_0)}\ni a_0\in\overline{\{f(\cS_n):n\in\bN\}}\subset\overline{U(\Vu_0)}$. 
 This is a contradiction. 
\end{proof}

\begin{acknowledgement}
The author thanks the referee for a very careful scrutiny and invaluable comments. This research was partially supported by JSPS Grant-in-Aid for Scientific Research (C), 15K04924.
The author also thank
the Research Institute for Mathematical Sciences, an International Joint Usage/Research Center located in Kyoto University.
\end{acknowledgement}


\def\cprime{$'$}

\end{document}